\documentclass[a4paper,12pt]{amsart}
\usepackage{amsmath,amsfonts,amssymb,amsthm}
\newcommand{\R}{\mathbb{R}}

\newcommand{\Z}{\mathbb{Z}}

\newcommand{\s}{\mathbb{S}}

\def\sgn{\mathop{\rm sgn}}
\def\Diff{\mathop{\rm Diff}}

\numberwithin{equation}{section}

\newtheorem{Proposition}{Proposition}[section]
\newtheorem{Theorem}[Proposition]{Theorem}
\newtheorem{Corollary}[Proposition]{Corollary}

\usepackage{graphicx}
\graphicspath{{./pictures/}}

\begin{document}
\author[Gaetano Zampieri]{Gaetano Zampieri}
\address{Universit\`a di Verona\\
Dipartimento di Informatica\\
Strada Le Grazie, 15\\
I-37134 Verona, Italy}
\email{gaetano.zampieri@univr.it}

\thanks{Supported by the PRIN 2007 directed by Fabio Zanolin.}



\title[Weak instability or isochrony]{Completely integrable Hamiltonian systems with weak Lyapunov instability
or isochrony}

\begin{abstract} The aim of this paper is to introduce a class of Hamiltonian autonomous systems in dimension 4
which are completely integrable and their dynamics is described in all details. They have an equilibrium point
which is stable for some rare elements of the class, and  unstable in most cases. Anyhow,  it is linearly stable
 (all orbits of the linearized system are bounded)
and no motion is asymptotic  in the past, namely no non-constant solution has the equilibrium 
as limit point as time goes to minus infinity. In the unstable cases, there is a
sequence of initial data which converges to the equilibrium point whose corresponding solutions are unbounded and the motion
is  slow. So instability  is quite weak and perhaps no such explicit examples of
instability are known  in the literature. The stable cases are also  interesting since  the level sets of the 2 first integrals independent and in involution keep being non-compact and stability is related to the isochronous periodicity  
of all  orbits near the equilibrium point and the existence of a further first integral. Hopefully, these superintegrable Hamiltonian systems
will deserve further research. 
\end{abstract}

\maketitle

\begin{center}
\emph{Dedicated to Angelo Barone Netto for his $75^{th}$ birthday}

\end{center} 

\section{Introduction}
\label{Introduction}

Let us have a quick look of our main results, the details and the proofs are in the paper.
We introduce the Hamiltonian system in $\R^4$
\begin{equation}\label{canonicalequations}
\begin{cases}
\dot q_1=p_2\\
\dot q_2=p_1\\
\dot p_1=-g'(q_1)\,q_2\\
\dot p_2=-g(q_1)
\end{cases}
\end{equation}
where $g$ is a $C^1$ function near $0$ and satisfies $g(0)=0$ and $g'(0)>0$.
The Hamiltonian function for the system is 
\begin{equation}\label{Hamiltonian}
H(q,p)=p_1\,p_2+g(q_1)\,q_2
\end{equation}
and  another first integral is
\begin{equation}\label{firstintegralK(q,p)}
 K(q,p)=\frac{p_2^2}{2}+V(q_1)\,,\quad 
V(q_1)=\int_0^{q_1}\, g(s)\,ds\,.
\end{equation}
The $q_1,p_2$-plane is invariant and the origin is a center on it, namely an open neighborhood $C$ of $(0,0)$ in this plane
is the union of periodic orbits which enclose $(0,0)$.  We restrict the phase space to
$M=\{(q_1,q_2,p_1,p_2)\in \R^4: (q_1,p_2)\in C\,, (q_2,p_1)\in\R^2\}$,
 so 
we have a global  center in the $q_1,p_2$-plane. The Hamiltonian vector fields
$\Omega\nabla H(q,p)$ in \eqref{canonicalequations} and 
$\Omega\nabla K(q,p)$
are both complete on $M$ (with $\Omega$ skew-symmetric matrix, 
see \eqref{ex4.4}).
The functions $H,K$  are in involution. The vectors $\nabla H(q,p), \nabla K(q,p)$ are  linearly independent at each point of the set 
$N:=\{(q,p)\in M:
(q_1,p_2)\ne (0,0)\}$ which is invariant for both the Hamiltonian vector fields $\Omega\nabla H$ and $\Omega\nabla K$. 
  
 A  nonempty  component $\Gamma\subset N$ of a level set of $H,K$ is diffeomorphic to $\s^1\times \R$ and there are coordinates $\phi$ mod $2 \pi$ and $z$ on $\;\s^1\times\R$ such that the differential
equations defined by the vector field $\Omega \nabla H$ on $\Gamma$ take the form
\begin{equation}\label{Liouville}
 \dot \phi=\omega\,,\qquad \dot z = v\qquad (\omega, v=\hbox{\sl const})\,.
 \end{equation}
If $(x_0,0)$ belongs to the  projection of $\Gamma$ on the $q_1,p_2$-plane and ${\mathcal T}(x_0)$ is the period
of the first component of any integral curve of $\Omega \nabla H$ on $\Gamma$, then $v=0$ if and only if ${\mathcal T}'(x_0)=0$.

A necessary and sufficient condition for the stability of the origin of \eqref{canonicalequations} is that  all periodic orbits 
in a neighborhood of $(0,0)$ in the $q_1,p_2$-plane have the same period so the center is (locally) isochronous. In this case all orbits of (1.1) in $\R^4$, with $(q_1,p_2)$ near $(0,0)$, are periodic and have the same period. Moreover,
a further first integral appears as we shall see using a result of Barone and Cesar~\cite{C}.
 This happens for instance for the following two functions (see \eqref{ex5.7} and \eqref{ex5.10})
\begin{equation}\label{2isochronousg}
 g(x)=1-\frac{1}{\sqrt{1+x}}\,,\qquad g(x)=1+x-\frac{1}{(1+x)^3}\,.
 \end{equation}

The corresponding Hamiltonian systems are superintegrable, see Fas\-s\`o~\cite{Fa},   `Maximally superintegrable systems' p.~110, and the Definition at p.~106. In Section~\ref{additional} we show other explicit examples of our superintegrable isochronous
systems, more precisely we find all those which have
 a third quadratic in the momenta first integral.

  Isochro\-nous systems had been recently studied starting from  powerful ideas
 of  Calogero by himself and collaborators, see~\cite{Ca}, Fran\c coise~\cite{F}, and the references therein.  There do not seem to be
 immediate connections of our systems with Calogero's techniques; however, perhaps the question deserves further
 work. 

In Section~\ref{Isochronous} there are results on 2 dimensional isochro\-nous centers taken from Zampieri~\cite{Z1}. We reproduce here the proofs to be self-contained. We 
relate the functions $g$, giving isochronous centers, with the even functions and so  isochrony, and  Lyapunov stability
for \eqref{canonicalequations}, is a rare phenomenon among the class \eqref{canonicalequations}. Isochronous centers were first studied by Urabe~\cite{U} with a different approach.  In~\cite{Z1} a new characterization by means of the so called `involutions' $h$ was found which permits to construct the  isochro\-nous centers. One of the Referees points out  that the survey
\cite{CS} shows my Theorem~\ref{ThIsochronous} as a classical
result and asks to clarify the paternity. Indeed a 1998 preprint  is quoted
 (see Theorem~7.3 in~\cite{CS}) instead of the 1989 paper~\cite{Z1}. However, Chavarriga \&  Sabatini's~\cite{CS} is a  good survey on general isochronous centers, not necessarily of the form considered in this paper (among the vast literature on this field see~\cite{M}).

Generally we have instability of the equilibrium for \eqref{canonicalequations} and it is very easy to show explicit functions for it (see Section~\ref{Isochronous} Corollary~\ref{necessaryconditions} and below),
for instance 
\begin{equation}\label{2unstableg}
 g(x)=\sin x\,,\qquad g(x)=\alpha x+\beta x^2+\gamma x^3\,,
 \end{equation}
with $\alpha>0, (\beta,\gamma)\ne(0,0)$. This kind of instability is quite weak since all orbits of the linearization of the system \eqref{canonicalequations} at the origin in $\R^4$ are bounded and there are no asymptotic motions to the equilibrium,
namely no non-constant solutions with the origin as limit point as $t\to-\infty$. In the present paper, instability without asymptotic motions is called \emph{weak instability}.

 In our Hamiltonian systems, instability occurs because there is a
sequence of initial data which converges to the origin whose corresponding solutions are unbounded and the motion
is slow, indeed   \eqref{Liouville} shows that the coordinate $z$ is an affine function of time.

Perhaps no such explicit examples of
instability are known  in the literature. The famous Cherry Hamiltonian system has a  linearly stable equilibrium point which is Lyapunov
unstable, however it has an asymptotic motion, see \cite{W} p. 412.


\section{Isochronous oscillations}
\label{Isochronous}

Let us start from the following equation 
\begin{equation}\label{scalarODE}
 \ddot{x}=-g(x)\,,\qquad \quad g(0)=0\,,\quad g'(0)>0\,,
 \end{equation}
where $g$ is continuous in a neighborhood of $0$ in $\R$ and differentiable at $0$  (in Section~\ref{Introduction} the function $g$ was $C^1$, while in the present Section~\ref{Isochronous} the existence of $g'(0)$ is enough).
This o.d.e. has the first integral of energy
\begin{equation}\label{energyforscalarODE}
G(x,\dot x)={\dot x^2\over 2}+V(x)\,,\qquad 
V(x)=\int_0^x\, g(s)\,ds\,.
\end{equation}
By means of this first  integral, we easily see that each Cauchy problem for \eqref{scalarODE} has a unique solution
if $g(x)=0$ only at $x=0$. 

The potential energy $V$ is a $C^1$ function and there exists
$\, V''(0)=g'(0)>0\,$. We can restrict the domain of $g$ to an open interval $J\ni 0$
 such that $V$ is strictly increasing on $J\cap \R_+$, strictly decreasing on $J\cap \R_-$,  and for each point $x\in J$ there is a unique point $h(x)\in J$ 
with
\begin{equation}
V\big(h(x)\big)=V(x),\qquad \sgn(h(x))=-\sgn(x),
\end{equation}
where $\sgn(x)$ is the sign of $x$. We check at once that the function 
\begin{equation}\label{u}
u(x):=\sgn(x)\sqrt{2V(x)}
\end{equation}
 is a  $C^1$ diffeomorphism 
onto the image $I=u(J)$ which  is a symmetric interval 
\begin{equation}\label{u'(0)}\begin{split}
&u'(0)=\sqrt{V''(0)},\quad u\in {\Diff}^1(J;I),  \\
 &0\in J,\quad 0\in I,\quad  y\in I\Longrightarrow -y\in I.
\end{split}\end{equation}
Moreover, since $h(x)=u^{-1}(-u(x))$ we have that $h$ is also a diffeomorphism:
\begin{equation}\label{h-properties}\begin{split} 
&h\in {\Diff}^1(J;J), \qquad h\big(h(x)\big)=x,\\
 &h(0)=0,\qquad\qquad\quad h'(0)=-1.
\end{split}\end{equation}
We call $h$ the \emph{involution} associated with $V$. Remark that the graph of $h$ is symmetric with respect to
the diagonal which intersects at the origin; indeed  $(x, h(x))$ has $(h(x),x)$ as symmetric point and this coincides
with the point $(h(x),h(h(x)))$ of the graph.

\begin{Theorem}[Determining isochronous centers]\label{ThIsochronous}   Let
$V$ be a $C^1$ function  near $0$ in $\R$ with $V(0)=V'(0)=0$, 
assume there exists $V''(0)>0$,  let $J$ be an open interval as above and let $h\in {\Diff}^1(J;J)$ be the involution associated
with $V$. Then all orbits of $\,\ddot{x}=-V'(x)\,$ 
which intersect the $J$ interval of the $x$-axis in the 
$x,\dot x$-plane, are periodic and enclose
$(0,0)$. Moreover, they all have the same period if and only if 
\begin{equation}\label{Visochronous}
V(x)={V''(0) \over 8}\,\big(x-h(x)\big)^2 \,,\qquad x\in J\,.
\end{equation}
In this case we say that the origin is an isochronous center for $\,\ddot{x}=-V'(x)\,$.
\end{Theorem}

Formula \eqref{Visochronous} corresponds to formula~(6.2) in~\cite{Z1}, the proof is included in the
proof of Proposition~1 in~\cite{Z1} as a particular case.

\begin{proof} By composition with the inverse of the diffeomorphism~\eqref{u}, the first integral~\eqref{energyforscalarODE} gives 
$\; 2\, G\big(u^{-1}(y),\dot x\big)=\dot x^2+y^2$. The first part of the thesis follows at once.

Now, consider 
a periodic orbit in the $x,\dot x$-plane which intersects the $x$-axis at $x_0\in J$,  $x_0>0$, then $h(x_0)<0$ is the other intersection.
By the  energy conservation we get the period of the orbit as
\begin{equation}\label{periodT(x)}
{\mathcal T}(x_0)=2\,\int^{x_0}_{h(x_0)} {{dx}\over{\sqrt{2\big(V(x_0)-V(x)\big)}}}\,.
\end{equation}
If  $y_0\in I$, $y_0>0$, $x_0=u^{-1}(y_0)$, then $h(x_0)=u^{-1}(-y_0)$, and the period 
\begin{equation}\begin{split}\label{periodT(y)}
&T(y_0):=
2\,\int^{u^{-1}(y_0)}_{u^{-1}(-y_0)} {{dx}\over{\sqrt{2\big(V(x_0)-V(x)\big)}}}=\\
&=2\,\int^{y_0}_{-y_0} {\left(u^{-1}\right)'(r)\, dr\over{\sqrt{y_0^2-r^2}}}
=2\,\int^{y_0}_{0} {\left(\left(u^{-1}\right)'(r )+\left(u^{-1}\right)'(-r )\right)\, dr
\over{\sqrt{y_0^2-r^2}}}\,. 
\end{split}\end{equation}
The change of integration variable $s=\arcsin (r/y_0)$ gives for   $y_0\in I$, $y_0>0$,
\begin{equation}\label{T}
T(y_0)=2\,\int^{\pi/2}_{0} 
\left(\left(u^{-1}\right)'(y_0\sin s )+\left(u^{-1}\right)'(-y_0\sin s )\right)ds.
\end{equation}
So
\begin{equation}
\lim_{y_0\to 0+} T(y_0)=2\pi \left(u^{-1}\right)'(0)=2\pi /\sqrt{V''(0)}\,.
\end{equation}

All  orbits  have the same period if and only if $T(y_0)$ equals this limit value for all $y_0\in I$, $y_0>0$. By \eqref{periodT(y)} this is equivalent to the following condition for all $z\in I$, $z>0$,
\begin{equation}\label{ex2.12}
{2\,\pi\over{\sqrt{V''(0)}}}=2\int^{z}_{0} {\left(\left(u^{-1}\right)'(r )+
\left(u^{-1}\right)'(-r )\right) dr
\over{\sqrt{z^2-r^2}}}\,. 
\end{equation}
We multiply both sides by $z/\sqrt{y^2-z^2}$ and we integrate from $0$ to  $y\in I$, $y>0$,
\begin{equation}\begin{split}
&{\pi\over{\sqrt{V''(0)}}}\,\int_0^y\;{z\, dz\over{\sqrt{y^2-z^2}}}=\\
&=\int_0^y\;{z\, dz\over{\sqrt{y^2-z^2}}}\,
\int^{z}_{0} {\left(\left(u^{-1}\right)'(r )+
\left(u^{-1}\right)'(-r )\right)\, dr
\over{\sqrt{z^2-r^2}}}\,.\end{split}
\end{equation}
By the change  of the integration order, see Figure~\ref{triangle}, we get
\begin{figure}
  \begin{center}
\includegraphics[scale=.5]{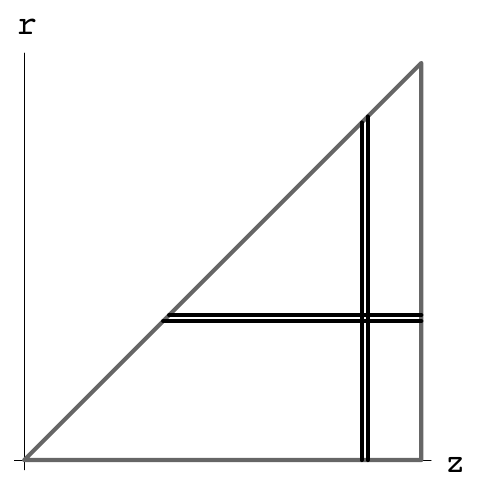}
\end{center}
\caption{Changing the  order of integration}
\label{triangle}
\end{figure}
\begin{equation}\begin{split}
{\pi\,y\over{\sqrt{V''(0)}}}=&\\
=\int_0^y\underbrace{\int_r^y{z\over{\sqrt{y^2-z^2}
\sqrt{z^2-r^2}}}\,dz}_{=\pi/2}&
 \left(\left(u^{-1}\right)'(r)+
\left(u^{-1}\right)'(-r)\right) dr
\end{split}\end{equation}
\begin{equation}\label{condizioneisocroniau(y)}
{2\,y\over{\sqrt{V''(0)}}}=
u^{-1}(y)-u^{-1}(-y),\qquad
y\in I,\; y>0.
\end{equation}
We deduced this condition from~\eqref{ex2.12}, and now we see at once that it implies~\eqref{ex2.12}.
By plugging $y=u(x)$ into~\eqref{condizioneisocroniau(y)} we have 
\begin{equation}\label{condizioneisocroniafinale}
{2\,u(x)\over{\sqrt{V''(0)}}}=
x-h(x),\qquad
x\in J,\; x>0.
\end{equation}
This condition is equivalent to~\eqref{Visochronous}.
\end{proof}

 In~\cite{Z1}, Section~6, there is the construction we are going to
show and also another one  (see from formula (6.3)).

From the theorem we easily get the following result.

\begin{Corollary}[Constructing isochronous centers]\label{constructionisochronous}  Let
$h:J\to J$ be a $C^1$ function  on an open interval  $J\subseteq \R$ containing $0$ which satisfies the conditions 
in~\eqref{h-properties}. Let $\omega>0$ and define
\begin{equation}\label{Vconstructionisochronous}
V(x)={\omega^2 \over 8}\,\big(x-h(x)\big)^2,\qquad x\in J.\end{equation}
Then there exists $V''(0)=\omega^2$ and 
 all orbits of $\,\ddot{x}=-V'(x)\,$ 
which intersect the $J$ interval of the $x$-axis in the 
$x,\dot x$-plane, are periodic and  have the same period $2 \pi/\omega$. To get $h$ as above, we can just
consider an arbitrary even $C^1$ function on a (symmetric) open interval  which vanishes at $0$, then  a $\pi/4$ clockwise rotation of its graph gives  a curve  containing an arc  $y=h(x)$ which satisfies~\eqref{h-properties} with $J$ open interval.  
\end{Corollary}

For instance,  starting from  $x\mapsto x^2/\sqrt{2}$ we can (even explicitly) calculate 
\begin{equation}\label{ex2.18}
h(x)=1+x-\sqrt{1+4x}\,,\qquad V(x)={\omega^2 \over 8}\,\left(-1+\sqrt{1+4x}\,\right)^2.
\end{equation}
Notice that the original quadratic function  is defined on the whole  $\R$, to get a function after rotation we throw out
an unbounded arc and obtain $(x,1+x-\sqrt{1+4x})$ with $x\ge -1/4$, finally, to satisfy (2.6) with $J$ open interval, we must restrict $x$ to $(-1/4,3/4)$ or to suitable smaller open intervals.

   Finally, a simple example is also
\begin{equation}\label{ex2.19}
h(x)=-{x\over{1+x}}\,,\qquad  V(x)={\omega^2 \over 8}\,x^2 \,\left({2+x\over{1+x}}\right)^2.
\end{equation}

Now, let us see some  necessary conditions to have constant period which can be used
to check whether a given function $V$, or its derivative $g$, locally gives an isochronous center or not.
In the sequel $V$ has as many derivatives as necessary. We saw how $h$ is related to an even function. The even functions have vanishing odd derivatives at $0$, so it will not be a surprise to see that the derivatives of the involution $h$ are not arbitrary. 

Consider the relation $h(h(x))=x$, perform 2
 derivations and calculate at $0$ taking into account $h(0)=0$ and $h'(0)=-1$. In this way we get an identity which shows that $h''(0)$ can take any value. By 3 derivations and calculating at $0$ we get
 $h^{(3)}(0)=-3 h''(0)^2/2$.
 Going further we see that the even derivatives are free while the odd ones are uniquely determined by
 the preceding derivatives. Indeed, differentiating  $n$ times we get $h^{(n)}(h(x))h'(x)^n+\dots+h'(h(x))h^{(n)}(x)=0$
 which is true for $n=2$ and is proved at once by induction for all $n\ge 2$, where the dots represent omitted terms which do not contain the highest order derivative $h^{(n)}$. So calculating at $x=0$ we have  $h^{(n)}(0)(-1)^n+\dots-h^{(n)}(0)=0$ and for even $n$ the $n$-th order derivative disappears while for odd $n$ we can solve for $h^{(n)}(0)$ in terms of the lower order derivatives.

 The first 5 terms in Taylor's formula are given by 
 \begin{equation}\label{ex2.20}
 h(x)=-x+a\, x^2-a^2\, x^3+
        b\, x^4+\left(2a^4 - 3 a b\right)\,x^5+ 
        o(x^{5})
        \end{equation}
        where $a,b$ are free parameters. Next, formula~\eqref{Visochronous} shows that the derivatives
              of the isochronous $V$ at $0$ are constrained to obey some conditions. By means of~\eqref{ex2.20}
              we can calculate $V^{(4)}(0)$ and $V^{(6)}(0)$  by the lower order
              derivatives and get 2 necessary conditions. Of course we can go forward to  infinite conditions.

\begin{Corollary}[Necessary conditions]\label{necessaryconditions} Let
$V$ admit $V^{(6)}(0)$ and satisfy  $V(0)=V'(0)=0$, $V''(0)>0$. Moreover, let the origin be an
isochronous center for $\,\ddot{x}=-V'(x)\,$. Then
\begin{equation}\label{ex2.21}
 V^{(4)}(0)={5 V^{(3)}(0)^2\over{3  V''(0)}}\,,\ V^{(6)}(0)={7 V^{(3)}(0) V^{(5)}(0)\over{V''(0)}}-{140 V^{(3)}(0)^4\over{9  V''(0)^3}}\,. 
 \end{equation}
 \end{Corollary}

To illustrate the necessary conditions let us use it to prove  the well known lack of isochronism of  the simple  pendulum 
$\ddot x + \sin x=0$
\begin{equation}\label{ex2.22}
  V^{(4)}(0)= \sin^{(3)}(0)=-1\ne 0= {5 \sin^{(2)}(0)^2\over{3  \sin'(0)}}={5 V^{(3)}(0)^2\over{3  V''(0)}}\,.
\end{equation}
Another example is 
\begin{equation}\label{ex2.23}
V(x)={\alpha\over 2}x^2+{\beta\over 3}x^3+{\gamma\over 4} x^4  \qquad (\alpha>0)
\end{equation}
for which the second condition in 
\eqref{ex2.21} implies $\beta=0$, and the first condition with $\beta=0$ gives $\gamma=0$ too.

We could also prove the above formulas in~\eqref{ex2.21} using the approach of   Barone, Cesar and Gorni~\cite{B} where the first 
2 derivatives of the period function ${\mathcal T}(x_0)$ at $0$ are computed with a procedure which carries on to higher-order
derivatives. With some regularity on $V$ they find the following formulas for $ V''(0)=1$
\begin{equation}\label{ex2.24}
{\mathcal T}'(0)=0\,,\qquad {\mathcal T}''(0)={\pi\over 4} \left({5\over 3}V^{(3)}(0)^2-V^{(4)}(0)\right)\,.
\end{equation}


\section{The dynamics in the Lagrangian framework} 
\label{lagrangian}

In this section we deal with the following 4-dimensional system defined by the Lagrangian function $L$
 \begin{equation}\label{ex3.1}\begin{split}
 & L(x,y,\dot x,\dot y)=\dot x\,\dot y-g(x)\,y,\\
 &\ddot{x}=-g(x),\quad \ddot{y}=-g'(x)y,\quad\qquad g(0)=0,\quad g'(0)>0,
 \end{split}\end{equation}
where $g\in C^1$ near $0$ in $\R$.
This system of differential equations has two first integrals $G(x,\dot x)$ and $F(x,y, \dot x, \dot y)$
\begin{equation}\label{ex3.2}\begin{split}
&G(x,\dot x)={\dot x^2\over 2}+V(x),\qquad
V(x)=\int_0^x\, g(s)\,ds,\\
 &F(x,y, \dot x, \dot y)=\dot y\,\dot x+g(x) y.
 \end{split}
\end{equation}
The first differential equation \eqref{ex3.1} separates and its dynamics was studied in the previous section. 
We restrict our attention to an open interval $J$ as in Theorem~\ref{ThIsochronous} and to the orbits in the
$x,\dot x$-plane which intersect $J$; their union is an open  neighborhood $C$ of $(0,0)$.
Let us fix $x_0\in J$, $x_0>0$, and denote by $t\mapsto X(t,x_0)$ the periodic solution  of
$\ddot{x}=-g(x)$ with $(x_0,0)$ as initial condition at time $0$. Next, we plug $X(t,x_0)$ into
the second differential equation in \eqref{ex3.1} and get the linear equation with periodic coefficient, Hill's equation,
\begin{equation}\label{ex3.3}
\ddot{y}=-g'\left( X(t,x_0)\right)\,y.
\end{equation}
The partial derivatives of $ X(t,x_0)$ give 2 independent solutions of \eqref{ex3.3} as one see by derivation
of the first equation for $X(t,x_0)$, in particular
\begin{equation}\label{ex3.4}\begin{split}
{\partial^2\over{\partial t^2}}{\partial X\over{\partial x_0}}(t,x_0)&=
{\partial\over{\partial x_0}}{\partial^2 X\over{\partial t^2}}(t,x_0)=\\
={\partial\over{\partial x_0}}\left(-g\left( X(t,x_0)\right)\right)&=-g'\left( X(t,x_0)\right){\partial X\over{\partial x_0}}(t,x_0).
\end{split}
\end{equation}
Let us define
\begin{equation}\label{ex3.5}
\phi(t)={\partial X\over{\partial x_0}}(t,x_0),\qquad \psi(t)=-{1\over{g(x_0)}}{\partial X\over{\partial t}}(t,x_0),
\end{equation}
so that for all $t$
\begin{equation}\label{ex3.6}\begin{split}
&\phi(0)=1,\quad \dot\phi(0)=0,\quad \psi(0)=0,\quad \dot\psi(0)=1,\\
 &\dot \psi(t)\,\phi(t)- \dot \phi(t)\,\psi(t)=1.
 \end{split}
\end{equation}
Let us denote by ${\mathcal T} (x_0)$, briefly $\tau$,  the period of $t\mapsto X(t,x_0)$, then $t\mapsto \phi(t+\tau)$ is also a solution of \eqref{ex3.3}, so a linear combination of $\phi$ and $\psi$, and 
\begin{equation}\label{ex3.7}
\phi(t+\tau)=\phi(\tau)\,\phi(t)+\dot \phi(\tau)\, \psi(t)=\phi(t)+\dot \phi(\tau)\, \psi(t)
\end{equation}
where we used $\phi(\tau)=1$  which comes from the last equality in~\eqref{ex3.6} with $t=\tau$ if we notice that
$\psi(\tau)=0$  and  $\dot\psi(\tau)=1$. Taking the derivative, calculating at $t=\tau$, and taking
into account  $\dot\psi(n\tau)=1$, we have
\begin{equation}\label{ex3.8}
\dot\phi(t+\tau)=\,\dot\phi(t)+\dot \phi(\tau)\,\dot \psi(t)\quad  \Longrightarrow\quad
\dot\phi\big((n+1) \tau\big)=\,\dot\phi(n\tau)+\dot \phi(\tau), 
\end{equation}
for all $n\in \Z$. By induction we have
\begin{equation}\label{ex3.9}
\dot\phi(n\tau)= n \dot \phi(\tau),\qquad\forall n\in \Z.
\end{equation}
If $\dot \phi(\tau)=0$ then $\bigl(\phi,\dot \phi\bigr)$ is periodic, otherwise it is unbounded.
A necessary and sufficient condition to have the former case is $\dot \phi(\tau)=0$ namely
\begin{equation}\label{ex3.10}
{\partial^2 X\over{\partial t\partial x_0}}\left({\mathcal T}(x_0),x_0\right)=0.
\end{equation}

To go ahead we need the differentiability of the period function

\begin{Proposition}\label{Pex3.1}  The period ${\mathcal T}(x_0)$ of  $t\mapsto X(t,x_0)$ is a $C^1$
function on $\{x_0\in J: x_0>0\}$.
\end{Proposition}

\begin{proof} By formula \eqref{T}, for $x_0\in J$, $x_0>0$, we have
\begin{multline}\label{ex3.11}
{\mathcal T}(x_0)=\\
=2\int^{\pi/2}_{0} 
\left(\left(u^{-1}\right)'(u(x_0)\sin s )+\left(u^{-1}\right)'(-u(x_0)\sin s )\right) ds.
\end{multline}
The function $u$ is a $C^1$ diffeomorphism as in the previous section. In the stronger hypothesis of this section, $g\in C^1$, we have that   $u(x)=\sqrt{V(x)}$ for $x>0$ is $ C^2$ as well as $V=g'$ and its inverse
$u^{-1}(y)$ for $y>0$. This proves the result.
\end{proof}

Now, we can differentiate $\partial_t X({\mathcal T}(x_0),x_0)=0$ with respect to $x_0$
\begin{equation}\label{ex.3.13}
 {\partial^2 X\over{\partial t^2}}\left({\mathcal T}(x_0),x_0\right)\,{\mathcal T}'(x_0)+{\partial^2 X\over{\partial x_0\partial t}}\left({\mathcal T}(x_0),x_0\right)=0.
 \end{equation}
We just remind that $(\partial^2 X/\partial t^2)\left({\mathcal T}(x_0),x_0\right)=-g(x_0)<0$ for $x_0>0$,
and have that  condition \eqref{ex3.10} is equivalent to
\begin{equation}\label{ex3.13}
{\mathcal T}'(x_0)=0.
\end{equation}
This is the condition in order $\phi$ to be periodic. Since the linear combination of $\phi$ and $ \psi$ gives the general solution, we have just proved the following

\begin{Proposition}\label{Pex3.2} Let us fix $x_0\in J$ with $x_0>0$, then the solutions to the 
${\mathcal T}(x_0)$-periodic  equation \eqref{ex3.2}  are either all ${\mathcal T}(x_0)$-periodic, or
all unbounded but those proportional to $t\mapsto {\partial_t X}(t,x_0)$. A necessary and sufficient condition for the former case is \eqref{ex3.13}.
\end{Proposition}

\begin{figure}
  \begin{center}
\includegraphics[scale=.5]{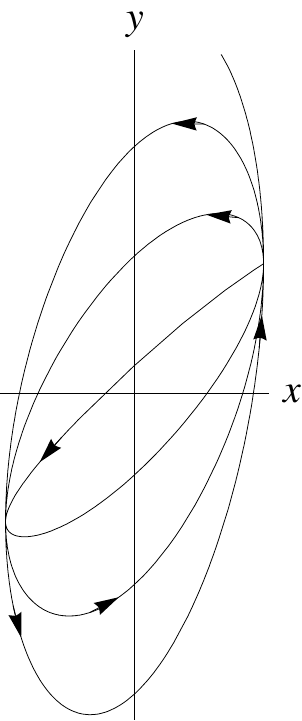}
\end{center}
\caption{Projection of an unbounded orbit}
\label{unbounded}
\end{figure}
In Figure~\ref{unbounded} the curve $(x(t),y(t))=(X(t,x_0), \phi(t))$, namely the solution to system of differential equations \eqref{ex3.1} with $x(0)=x_0>0$, $y(0)=1$, $\dot x(0)=\dot y(0)=0$, in a non-periodic  case. The initial point $(x_0,1)$ and the point 
$\big(h(x_0),g(x_0)/g(h(x_0))\big)$ are crossed at any period (as one can see from \eqref{ex3.6} at $t=\tau/2$).

Now,  notice that the origin in $\R^4$ is an unstable
equilibrium for \eqref{ex3.1} if and only if we can find $x_0>0$ arbitrarily close to $0$ such that equation
\eqref{ex3.2} has unbounded solutions. So by Proposition~\ref{Pex3.2} we have

\begin{Theorem}[Stability and weak instability]\label{exT3.3}  Let $g$ be  a $C^1$ function  near $0$ in $\R$ with $g(0)=0$, $g'(0)>0$,
then the origin in $\R^4$ is a stable equilibrium for the system~\eqref{ex3.1}, if and only if the origin in $\R^2$
is  a   locally isochronous center for  $\ddot x=-g(x)$, in this case all orbits of \eqref{ex3.1} with $(x,\dot x)$ near $(0,0)$ are periodic and have the same period. If the equilibrium is unstable, then there is a
sequence of initial data which converges to the origin whose corresponding solutions are unbounded. The
instability is weak, by this we mean that there are no asymptotic motions to the equilibrium,
namely no non-constant solutions which have the origin as limit point as $t\to-\infty$.
\end{Theorem}

We already proved the first part of the statement, for the last, let us only remark
that the distance of an orbit from the origin in $\R^4$ is greater than the distance of the projection in the
$x,\dot x$-plane which is strictly positive unless the projection is $(0,0)$ and in this case
the second differential equation in \eqref{ex3.1} gives the harmonic oscillator $\ddot y=-g'(0)y$. An explicit function $g$
to give such an example is $\sin(x)$ (see \eqref{ex2.22}). However, as we know from Section~\ref{Isochronous}, almost all functions
$g$ as in the statement of the theorem, give instability. 

Examples of the rare functions which give stable equilibria are (see \eqref{ex2.18} and \eqref{ex2.19})
\begin{equation}\label{ex3.14}
g(x)={1\over 2}\left(1-{1\over{\sqrt{1+4x}}}\right),\qquad g'(0)=1,
\end{equation}
\begin{equation}\label{ex3.15}
g(x)={1\over 4}\left(1+x-{1\over{(1+x)^3}}\right),\qquad g'(0)=1.
\end{equation}
 For all these functions which give stable equilibria, at least whenever $g\in C^2$, the system \eqref{ex3.1} admits a  Lyapunov function which is a further first integral and it is positive definite near the origin. 
 This additional first integral  is smooth in a neighborhood of the origin in $\R^4$ and at least
continuous at the origin as proved by Barone and Cesar in~\cite{C}. In the following statement $C\subseteq \R^2$
is the set defined  below formula~\eqref{ex3.2}.

\begin{Theorem}[Lyapunov functions]\label{Tex3.4} Let $g:J\to\R$ be  a $C^2$ function on the open interval $J\subseteq\R$ with $0\in J$, $g(0)=0$, $g'(0)>0$. Suppose that
 all orbits of $\,\ddot{x}=-g(x)\,$ 
which intersect the $J$ interval of the $x$-axis in the 
$x,\dot x$-plane, are periodic and  have the same period $2 \pi/\sqrt{g'(0)}$, so
the origin in $\R^4$ is a stable equilibrium for the system~\eqref{ex3.1}. Then  there exists
\begin{equation}\label{ex3.16}
E(x,y, \dot x, \dot y)=a(x,\dot x) \dot y^2+b(x,\dot x) y\dot y+c(x,\dot x)  y^2,\     \       a,b,c:C\to \R,
\end{equation}
continuous and positive definite function, which is a  (global) first integral for the system~\eqref{ex3.1}.
\end{Theorem}

The first integral is obtained by means of the functions in~\eqref{ex3.5} and a suitable inverse function, see~\cite{C} for details.

\begin{Proposition}[Eigenvalues]\label{Pex3.5} Let $g$ be  a $C^1$ function  near $0$ in $\R$ with $g(0)=0$, $g'(0)>0$,
then the linearization of the system~\eqref{ex3.1} at the origin in $\R^4$ has the double eigenvalues $\pm i\sqrt{g'(0)}$.
\end{Proposition}

Let us remark that we can easily construct unstable cases with some $x_0$ at which~\eqref{ex3.13}
holds and so the corresponding orbits are all periodic. If we want the period
function composed with $u^{-1}$ to be for instance $T(y_0)=2\pi (1-y_0^2+y_0^4)$, whose derivative 
vanishes at $y_0=1/\sqrt{2}$ we integrate
\begin{equation}\label{ex3.17}
\int_0^{y_0}{z T(z)\over{\sqrt{y_0^2-z^2}}}dz=2\pi\left(y_0-{2\over 3}y_0^3+{8\over{15}}y_0^5\right).
\end{equation}
So we get the relation
\begin{equation}\label{ex3.18}
 u^{-1}(y_0)-u^{-1}(-y_0)  =2\left(y_0-{2\over 3}y_0^3+{8\over{15}}y_0^5\right).
 \end{equation} 
 The symmetric choice  $u^{-1}(-y_0)=-u^{-1}(y_0)$ gives
 \begin{equation}\label{ex3.19}
 u^{-1}(y_0) =y_0-{2\over 3}y_0^3+{8\over{15}}y_0^5.
 \end{equation} 
 By inversion we get $u(x)$, then $V(x)=u(x)^2/2$ and finally $g(x)=V'(x)$. There is an isolate $x_0$
 at which~\eqref{ex3.13}  holds.  We can also imagine more complicate examples where these $x_0$ accumulate at $0$.   
 
 Finally, let us mention that most of the results in Section~\ref{lagrangian} were found in~\cite{Z2} by a different 
 (more complicated) proof which
 included other differential equations in a family for which the present system was a particular case. 
 However, in~\cite{Z2} I had not realized the Lagrangian character of the present case which is very important
 to our purposes and will be exploited in the next section.

\section{Complete integrability} 

In this section we write $q=(q_1,q_2)=(x,y)$. The Legendre transformation takes the Lagrangian system \eqref{ex3.1} into
\begin{equation}\label{ex4.1}\begin{split}
& H(q,p)=p_1\,p_2+g(q_1)\,q_2,\\
&\dot q_1={\partial H\over{\partial p_1}}(q,p)=p_2,\qquad\qquad\qquad 
\dot q_2={\partial H\over{\partial p_2}}(q,p)=p_1,\\
&\dot p_1=-{\partial H\over{\partial q_1}}(q,p)=-g'(q_1)\,q_2,\qquad\dot p_2=-{\partial H\over{\partial q_2}}(q,p)=-g(q_1).
\end{split}\end{equation}
The Hamiltonian function corresponds to the first integral $F$ while the first integral $G$ becomes
\begin{equation}\label{ex4.2}
K(q,p)={p_2^2\over 2}+V(q_1)\,,\quad 
V(q_1)=\int_0^{q_1}\, g(s)\,ds\,.\end{equation}
Our phase space is 
\begin{equation}\label{ex4.3}
M=\{(q_1,q_2,p_1,p_2)\in \R^4: (q_1,p_2)\in C\,, (q_2,p_1)\in\R^2\}\end{equation}
where $C\subseteq \R^2$ is an open set mentioned in Section~\ref{lagrangian}  between formulas \eqref{ex3.2} and \eqref{ex3.3} so 
to have a  \emph{global center} in the $q_1,p_2$-plane. The Hamiltonian vector fields
\begin{equation}\label{ex4.4}
\Omega\nabla H(q,p)=\begin{pmatrix}0 & 0 & 1 & 0\\ 0 &0 &0 &1\\ -1 &0 &0 &0\\ 0 &-1 &0 &0\end{pmatrix}
\begin{pmatrix}\partial_{q_1} H(q,p)\\
\partial_{q_2} H(q,p)\\
\partial_{p_1} H(q,p)\\
\partial_{p_2} H(q,p)\end{pmatrix}=
\begin{pmatrix}p_2\\ p_1\\-g'(q_1)q_2\\-g(q_1)\end{pmatrix}
\end{equation}
\begin{equation}\label{ex4.5}
\Omega\nabla K(q,p)=\begin{pmatrix}0 & 0 & 1 & 0\\
 0 &0 &0 &1\\
  -1 &0 &0 &0\\
   0 &-1 &0 &0\end{pmatrix}
\begin{pmatrix}\partial_{q_1} K(q,p)\\
\partial_{q_2} K(q,p)\\
\partial_{p_1} K(q,p)\\
\partial_{p_2} K(q,p)\end{pmatrix}=
\begin{pmatrix}0\\ p_2\\-g(q_1)\\ 0\end{pmatrix}
\end{equation}
are both \emph{complete} on $M$, namely their integral curves are all defined on the whole $\R$. Indeed,
this is clear for \eqref{ex4.4} since we have a global center on the $q_1,p_2$-plane and linear equations
in $q_2,p_1$;
for \eqref{ex4.5} we simply integrate and remind \eqref{ex4.3}
\begin{equation}\label{ex4.6}\begin{split}
&q_1(t)=q_1(0),\qquad  q_2(t)=q_2(0)+p_2(0)\,t,\\
  &p_1(t)=p_1(0)-g(q_1(0))\, t,\qquad
p_2(t)=p_2(0).\end{split}
\end{equation}

The functions $H,K$ are \emph{in involution}, indeed their Poisson brackets vanish:
\begin{equation}\label{ex4.7}\begin{split}
&\{H,K\}={\partial H\over{\partial q}}\cdot {\partial K\over{\partial p}}-{\partial K\over{\partial q}}\cdot {\partial H\over{\partial p}}=\\
&=\begin{pmatrix}g'(q_1)q_2\\
 g(q_1)\end{pmatrix}\cdot \begin{pmatrix}0\\
  p_2\end{pmatrix}-\begin{pmatrix}g(q_1)\\ 0\end{pmatrix}\cdot \begin{pmatrix}p_2\\ p_1\end{pmatrix}=0\,.\end{split}
\end{equation}

The vectors $\nabla H(q,p), \nabla K(q,p)$ are \emph{linearly independent} at each point of the set 
\begin{equation}\label{ex4.8}
N:=\{(q,p)\in M:
(q_1,p_2)\ne (0,0)\}
\end{equation} which is \emph{invariant} for both the Hamiltonian vector fields $\Omega\nabla H$ and $\Omega\nabla K$. Indeed,
 the condition $\alpha \nabla H(q,p)+\beta \nabla K(q,p)=0$   implies 
 $\alpha g(q_1)=0$ and $\alpha p_2=0$, so $\alpha=0$ since  $(q_1,p_2)\ne (0,0)$; then it also implies
  $\beta g(q_1)=0$ and $\beta p_2=0$ so $\beta=0$.
  
  Let $\Gamma\subset N$ be a  nonempty  component of a level set of $(H,K)$, then we see at once that it is not
  compact since it contains the unbounded curve $p_1\mapsto (q_1, c/g(q_1),p_1,0)$ where $c$ is the value of $H$
  and  $V(q_1)\ne 0$ is the one of $K$.
  By a well known theorem (see Theorem~3, Chapter~4, in Arnold, Kozlov and Neishtadt~\cite{A}), $\Gamma$ is diffeomorphic to $\s^1\times \R$.
  More precisely

\begin{Theorem}[Complete integrability]\label{Tex4.1} Let $g$ be  a $C^1$ function  near $0$ in $\R$ with $g(0)=0$, $g'(0)>0$. The functions $H,K: M\to \R$, in \eqref{ex4.1} and \eqref{ex4.2}, are in involution on $M$, ${\rm dim}\, M= 4$, and the Hamiltonian vector fields  $\Omega\nabla H$ and $\Omega\nabla K$ are complete on $M$. The set $N$ in \eqref{ex4.8}
is invariant for $\Omega\nabla H$ and $\Omega\nabla K$ and  $H,K$ are independent on $N$.
Let $\Gamma\subset N$ be a  nonempty  component of a level set of $(H,K)$ then it is diffeomorphic to $\;\s^1\times \R$ and there are coordinates $\phi$ mod $2 \pi$ and $z$ on $\;\s^1\times\R$ such that the differential
equations defined by the vector field $\Omega \nabla H$ on $\Gamma$ take the form
\begin{equation}\label{ex4.9}
\dot \phi=\omega\,,\qquad \dot z = v\qquad (\omega, v=\hbox{\sl const})\,.\end{equation}
If $(x_0,0)$ belongs to the  projection of $\Gamma$ on the $q_1,p_2$-plane and ${\mathcal T}(x_0)$ is the period
of the first component of any integral curve of $\Omega \nabla H$ on $\Gamma$, then $\omega=2\pi/{\mathcal T}(x_0)$ and 
$v=0$ if and only if ${\mathcal T}'(x_0)=0$.
\end{Theorem}

  We already proved everything in the statement of the theorem but the last sentence which is a 
  straight consequence of   ${\mathcal T}'(x_0)=0$ being the necessary and sufficient condition for the
  solutions on $\Gamma$ to be all periodic (see Proposition~\ref{Pex3.2}), and  \eqref{ex4.9} gives periodic solutions if and     only if $v=0$.
  
 For the \emph{isochronous} 
  systems, namely whenever
  the period function is constant and all integral curves of $\Omega \nabla H$ in $M$ have the same period,
  we know from Theorem~\ref{Tex3.4} that a further first integral exist  which is quadratic in $q_2,p_1$. 
  Of course we do not expect this first integral to be in involution with $K$. In the
  next section we classify the  cases where a third first integral quadratic in the
  momenta $p_1,p_2$ exits and we arrive at explicit formulas for it. They are all isochronous.
     
   
\section{Explicit superintegrable systems} \label{additional}

 Let $g$ be  a $C^1$ function  near $0$ in $\R$ with $g(0)=0$, $g'(0)>0$. Suppose that the system \eqref{ex4.1} has a first integral $W$ of the 
form
\begin{equation}\label{ex5.1}
A(q_1,q_2) p_1^2+B(q_1,q_2)p_1 p_2+C(q_1,q_2)p_2^2+U(q_1,q_2)\,.
\end{equation}
Then the equation $\{H,W\}=0$ give a cubic polynomial in $p_1,p_2$. The coefficients 
vanish if and only if
\begin{equation}\label{ex5.2}
 \partial_2 A=0,\qquad \partial_1 C=0,\qquad \partial_1 A+ \partial_2 B=0,\qquad \partial_1 B +\partial_2 C=0,
 \end{equation}
 \begin{equation}\label{ex5.3}\begin{split}
 \partial_1 U(q_1,q_2)=2 C(q_1,q_2) g(q_1)+B(q_1,q_2) g'(q_1)q_2,\\
 \partial_2 U(q_1,q_2)=2 A(q_1,q_2) g'(q_1)q_2+B(q_1,q_2) g(q_1).
 \end{split}\end{equation}
The conditions in \eqref{ex5.2} are the same as those in (3.2.2) in Hietarinta~\cite{H} where the celebrated Darboux problem is treated
(only  the order of the coordinates is changed since $\dot q_1=p_2$,  $\dot q_2=p_1$). They have the
following solution where $a,b_1,b_2,c_1,c_2,c_3\in\R$ are arbitrary 
\begin{equation}\label{ex5.4}\begin{split}
& A(q_1,q_2)=a q_1^2+b_1 q_1+c_1\\
           & B(q_1,q_2)=-2 a q_1 q_2-b_1 q_2-b_2 q_1+c_3\\
           & C(q_1,q_2)=a q_2^2+b_2 q_2+c_2
           \end{split}\end{equation}
Plugging \eqref{ex5.4} in \eqref{ex5.3} we get the following condition for the integrability of $U$
\begin{multline}\label{ex5.5}
3(b_2 + 2 a q_2) g(q_1)=\\ 
   =3(b_1 + 2 a q_1)q_2g'(q_1)+
     2(c_1 + q_1(b_1 + a q_1))q_2
       g''(q_1)\,.\end{multline}
       For $q_2=0$ we get $3 b_2 g(q_1)=0$ which implies $b_2=0$, and back to \eqref{ex5.5}  we have
       \begin{equation}\label{ex5.6}
       6 a  g(q_1)= 
   3(b_1 + 2 a q_1)g'(q_1)+
     2(c_1 + q_1(b_1 + a q_1))
       g''(q_1)\,.\end{equation}
       For $a=0$, $g(0)=0$ and $g'(0)=\omega^2$, this equation gives
       \begin{equation}\label{ex5.7}
       g(q_1)={\omega^2\over\lambda}\left(1-{1\over{\sqrt{1+2\lambda q_1}}}\right)\end{equation}
       where $\lambda=b_1/(2c_1)$ (for $c_1=0$ there are no solutions). The values $\lambda=2$ and $\omega=1$ give~\eqref{ex3.14}. 
       The first integral \eqref{ex5.1} for this $g$ is easily  obtained with $b_1=2\lambda c_1$, $a=0$, $b_2=0$
       \begin{multline}\label{ex5.8}
       c_1 p_1^2 (1 +2\lambda q_1) + p_1 p_2(c_3 - 2\lambda c_1 q_2)+c_2 p_2^2+\\
      +c_2 {\omega^2\over{\lambda^2}}\left(-1+\sqrt{1+2\lambda q_1}\right)^2 
      +c_1\omega^2q_2^2+q_2 \left(c_3-2\lambda c_1 q_2\right)g(q_1)+d
      \end{multline}  
      where $d\in \R$. For $c_1=c_3=d=0$, $c_2=1/2$, we have the first integral $K$ we already know, 
      while $H$  corresponds to the choice $c_1=c_2=d=0$, $c_3=1$. 
      For $c_1=c_2=1$, $c_3=d=0$ we have the first integral
      \begin{equation}\label{ex5.9}
      p_1^2 (1 +2\lambda q_1)  - 2\lambda q_2 p_1 p_2+ p_2^2+
      2 V(q_1) +\left(\omega^2-2\lambda  g(q_1)\right)q_2^2\,.\end{equation} 
      This function is positive definite near the origin of $\R^4$ as we check at once. This is why we have Lyapunov stability of the equilibrium
      and compact orbits.
      
      We dealt with the particular case $a=0$. In the sequel $a\ne 0$ and we fix $a=1$ dividing the first integral
      by $a$. The coefficient of $g''$ in equation \eqref{ex5.6} is a quadratic function in $q_1$ with discriminant
      $b_1^2-4 c_1$. Let us consider first the case where this vanishes, namely $c_1=b_1^2/4$. If $b_1=0$ then \eqref{ex5.6}
     with $g(0)=0$, $g'(0)=\omega^2$, gives the trivial function $g(x)=\omega^2 x$, while for  $\lambda:=2 b_1\ne 0$ we get the following solution with $g(0)=0$, $g'(0)=\omega^2$:
      \begin{equation}\label{ex5.10}
            g(q_1)=\frac{\omega^2}{4}\left(\lambda+q_1-\frac{\lambda^4}{(\lambda+ q_1)^3}\right)
     \end{equation}
      Notice that this formula includes \eqref{ex3.15} for $\omega=1$ and $\lambda=1$. 
      The first integral \eqref{ex5.1}, which corresponds to \eqref{ex5.10},  with $c_3=0$, $c_2=1$, and which vanishes at the origin of $\R^4$, is 
  \begin{multline}\label{ex5.12}
   p_1^2 (\lambda +  q_1)^2 -2 p_1 p_2 (\lambda +  q_1)q_2
      + p_2^2(1 + q_2^2)+\\
      +\frac{\omega^2}{4}\left((\lambda+q_1)^2+\lambda^4\frac{1+4 q_2^2}{(\lambda+q_1)^2}\right)-\frac{\lambda^2\omega^2}{2}\,.
      \end{multline}
       We can check that at the origin of $\R^4$  its gradient vanishes and the Hessian matrix  is the diagonal
   matrix $2(\omega^2,\lambda^2\omega^2,\lambda^2,1)$, so we have 
again a (positive definite near the origin) Lyapunov function for all values of the parameter $\lambda\ne 0$.

Now we are ready to deal with  $a=1$ and $b_1^2-4 c_1\ne 0$, we also consider  $c_1\ne 0$  since  $c_1=0$ gives no solutions. Then the solution 
of the equation \eqref{ex5.6} with $g(0)=0$ and $g'(0)=\omega^2>0$ is
\begin{multline}\label{ex5.13}
g(q_1)={2 c_1\omega^2 \over{(b_1^2 - 4 c_1)^2}}\Bigg((b_1^2+4c_1)(b_1 + 2 q_1) +\\ +b_1{b_1^2-4 c_1-2
  \big(b_1 + 2 q_1\big)^2\over{{\sqrt{1 +q_1(b_1+q_1)/c_1}}}}\Bigg).\end{multline}
      The first integral \eqref{ex5.1} with $c_3=0$, $c_2=c_1$,  divided by $c_1$, and which vanishes at the origin of $\R^4$, is
    \begin{multline}\label{ex5.15}
    p_1^2+p_2^2+\frac{1}{c_1}\bigl(\,p_1q_1-p_2q_2\bigr)\bigl(\,p_1(b_1+q_1)-p_2q_2\bigr)+\\
    +\frac{\omega^2}{(b_1^2-4 c_1)^2}\Bigl(8 b_1^2c_1^2+4c_1\big(b_1^2+4c_1\big)(b_1+q_1)q_1+
    \big(16c_1^2-b_1^4\big)q_2^2\Bigr)+\\
    +\frac{2b_1 \omega^2}{(b_1^2-4 c_1)^2}\frac{(b_1+2q_1)\Bigl(-4c_1\bigl(c_1+(b_1+q_1)q_1\bigr)+\bigl(b_1^2-4c_1\bigr)q_2^2\Bigr)}{\sqrt{1+q_1(b_1+q_1)/c_1}}\,.
    \end{multline}
   We can check that at the origin of $\R^4$  its gradient vanishes and the Hessian matrix  is the diagonal
   matrix $2(\omega^2,\omega^2,1,1)$, so we have a positive definite  function near the origin, this yields Lyapunov stability of the equilibrium
      and compact orbits.

\bigbreak   \noindent   
\emph{Acknowledgements.} I thank the Referees for the  thorough, constructive and helpful comments and suggestions on the manuscript. The picture in Section~\ref{lagrangian} was  made using the   application {\it
Mathematica} by    Wolfram Research Inc. by means of the   package
CurvesGraphics6  by
Gianluca Gorni.


\bibliographystyle{plain}

\end{document}